\documentclass[12pt,reqno]{amsart}
\usepackage{amssymb,amsmath,amsthm,fullpage,enumerate}
\usepackage{xcolor}
\usepackage{biblatex} 
\usepackage{comment}
\usepackage{cases}

\usepackage[OT2,OT1]{fontenc}

\DeclareFontFamily{OT2}{cmr}{\hyphenchar\font45 }
\DeclareFontShape{OT2}{cmr}{m}{n}{<->wncyr10}{}
\DeclareFontShape{OT2}{cmr}{m}{it}{<->wncyi10}{}
\DeclareFontShape{OT2}{cmr}{m}{sc}{<->wncysc10}{}
\DeclareFontShape{OT2}{cmr}{b}{n}{<->wncyb10}{}
\DeclareFontShape{OT2}{cmr}{bx}{n}{<->ssub*wncyr/b/n}{}

\DeclareFontFamily{OT2}{cmss}{\hyphenchar\font45 }
\DeclareFontShape{OT2}{cmss}{m}{n}{<->wncyss10}{}

\DeclareRobustCommand\cyr{\fontencoding{OT2}\selectfont}
\DeclareTextFontCommand{\textcyr}{\cyr}

\DeclareUnicodeCharacter{0306}{}

\newtheorem{theorem}{Theorem}[section]
\newtheorem{prop}[theorem]{Proposition}

\newtheorem{corr}[theorem]{Corollary}

\theoremstyle{definition}

\newtheorem{deff}[theorem]{Definition}

\theoremstyle{remark}
\newtheorem{comm}[theorem]{Remark}

\newcommand{\R}{\mathbb R}

\newcommand{\C}{\mathbb C}
\newcommand{\Cplus}{{\mathbb C}_+}
\newcommand{\Cminus}{{\mathbb C}_-}

\DeclareMathOperator{\supp}{supp}

\newcommand{\p}{\partial}

\newcommand{\z}{\bar z}

\DeclareMathOperator{\im}{Im}
\DeclareMathOperator{\re}{Re}

\newcommand{\Har}{H^p(D)}

\newcommand{\dc}{\mathcal{D}'(\p D)}

\title{The Schwarz boundary value problem for boundary values in the sense of distributions}

\author{William L. Blair}
\address{Department of Mathematical Sciences\\
  University of Arkansas\\
  Fayetteville, Arkansas}
\email{wlblair@uark.edu}

\keywords{Schwarz boundary value problem, boundary value in the sense of distributions, nonhomogeneous Cauchy-Riemann equation}

\subjclass[2010]{30E25, 35F30, 46F20, 35F15, 30J99, 46F99}

\begin{document}

\begin{abstract}
    We construct solutions to the Schwarz boundary value problem on the unit disk and the upper half-plane when the boundary condition is with respect to boundary values in the sense of distributions. 
\end{abstract}

\maketitle

\section{Introduction}
In this paper, we extend the classes of boundary conditions under which the Schwarz boundary value problem is solvable. The Schwarz boundary value problem is a classically studied boundary value problem in the setting of complex-valued partial differential equations. The problem is to find a holomorphic function on a domain in the plane that has real part which agrees with a prescribed function on the boundary of the domain. 

When considered on the unit disk or upper half-plane, this problem is solvable by considering the Dirichlet problem with the same boundary condition, i.e., finding a real-valued harmonic function on the domain which agrees with the prescribed boundary condition. For the unit disk and the upper half-plane, the Dirichlet problem is solved by the Poisson integral of the boundary condition, for many nice classes of boundary condition. Once this harmonic function is found, since we are working on a domain, it follows that the harmonic function has a harmonic conjugate. On the disk and upper half-plane, the harmonic conjugate is known and well-defined up to an imaginary constant. However, this constant makes the boundary value problem not well-defined, as there is an infinite number of solutions parameterized by the imaginary constant. This issue is resolved by specifying the constant. Since the harmonic conjugate of the Poisson integral is equal to this constant when evaluated at $z = 0$ on the disk and $z = i$ on the upper half-plane, we can specify the constant by choosing the value of the harmonic conjugate at this point evaluation. Once the Schwarz boundary value problem is amended to include this pointwise condition, the problem is well-defined. 

In \cite{BegBook}, H. Begehr explicitly solves a generalized Schwarz boundary value problem in the unit disk where, instead of requiring holomorphicity in the interior of a solution, a solution satisfies a nonhomogeneous Cauchy-Riemann equation in the interior. In \cite{Beg}, H. Begehr solves a generalized Schwarz boundary value problem on the unit disk where solutions satisfy $n$th-order nonhomogeneous Cauchy-Riemann equations, $n \geq 1$, (solutions to these higher-order Cauchy-Riemann equations are called polyanalytic in the special case of the equation being homogeneous, see \cite{Balk}), where a solution of the boundary value problem is also required to have derivatives with real parts that satisfy boundary conditions and imaginary parts that satisfy pointwise evaluation conditions for all orders up to $n-1$. In \cite{Gaertner}, E. Gaertner proves the analogue to Begehr's first-order result on the upper half-plane, and in \cite{bvpuhp}, A. Chaudhary and A. Kumar generalize E. Gaertner's result to the higher-order case in the upper half-plane.

In all of these previously known results concerning the generalized Schwarz boundary value problem, the boundary conditions are with respect to continuous functions, with additional Lebesgue integrability constraints in the case of the upper half-plane which are superfluous in the unit disk case. We extend the solvable classes of the Schwarz boundary value problem to classes with boundary conditions which are in terms of a boundary value in the sense of distributions of holomorphic functions with tempered growth at the boundary. An intermediary step in this progress to weaker boundary conditions is realized in \cite{WB} for the unit disk, where the generalized Schwarz boundary value problem is solved when the boundary condition is in terms of a boundary value in the sense of distributions of a function in certain holomorphic Hardy spaces. The results of this paper extend the corresponding results in \cite{WB}. In particular, there exist holomorphic functions on the unit disk which have the property of tempered growth at the boundary but are not in any of the holomorphic Hardy spaces (see \cite{GHJH2} and \cite{Duren}), so no appeal to Hardy space results is needed in the first-order case. Also, the upper half-plane case with boundary condition being a boundary value in the sense of distributions was not considered in \cite{WB} and is new to this paper. 

We make the extension in the disk case by appealing to Theorem 3.1 from \cite{GHJH2} which allows us to both identify the classes of holomorphic functions defined on the disk with boundary values in the sense of distributions and provides a constructive method for representing them given their boundary value in the sense of distributions. A similar technique is employed in the upper half-plane case. 

We describe the layout of this paper. In the next section, we consider the case of the unit disk. Preliminary definitions and results are presented concerning associated function spaces and certain integral operators that will be used, initially for general objects and followed by those associated with boundary values in the sense of distributions. Next, we define the generalized Schwarz boundary value problem that we consider and explicitly show its solution. Next, we describe the higher-order version of the generalized Schwarz boundary value problem, and we make an extension of the corresponding result in \cite{WB} with respect to the zero-order boundary condition. We conclude our study of the Schwarz boundary value problem on the unit disk by considering a special case with respect to the nonhomogeneous part of the corresponding differential equation. The significance of the special case is that the nonhomogeneous part of the differential equation is not assumed to be integrable. This is unique in the literature. We extend the first-order special case result to one for the higher-order Schwarz problem where every boundary condition is in terms of boundary values in the sense of distributions, with no appeal to the theory of the holomorphic Hardy spaces. The last section considers the case of the upper half-plane using the same structure as the previous section. After preliminary definitions and results specific to this domain, we define and solve the generalized Schwarz boundary value problem on the upper half-plane with boundary condition a boundary value in the sense of distributions. Finally, we utilize the first-order generalized Schwarz boundary value problem result to extend the corresponding higher-order result in \cite{bvpuhp} with respect to the zero-order boundary condition.

We thank Professor Andrew Raich and Professor Gustavo Hoepfner for their support during the time this work was produced. 

\section{The case of the unit disk}

\subsection{Definitions and Results}

Throughout we will let $D(0,r)$ denote the disk centered at the origin with radius $r$ in the complex plane, $D := D(0,1)$  is the unit disk in the complex plane, $\p D(0,r)$  and $\p D$ denote the boundary of $D(0,r)$ and $D$ respectively, $H(D)$ denote the space of holomorphic functions defined on $D$, $H^p(D)$ the classical holomorphic Hardy space of functions on $D$, $C^\infty(\p D)$ is the space of smooth functions defined on $\p D$, and $\dc$ is the space of distributions on $\p D$.

First we introduce some integral operators and results associated with them.

\begin{deff}\label{vekopondisk}
For $f: \mathbb{C} \to \mathbb{C}$ and $z  \in \mathbb{C}$, we denote by $T_D(\cdot)$ the integral operator defined by 
\[
T_D(f)(z) = -\frac{1}{\pi} \iint_D \frac{f(\zeta)}{\zeta - z}\, d\xi\,d\eta,
\]
whenever the integral is defined.
\end{deff}

\begin{theorem}[Theorem 1.26 \cite{Vek} p.\,47] \label{Vekonepointtwentysix}
If $f \in L^q(D)$, $1 \leq q \leq 2$, then $T_D(f) \in L^\gamma(D)$, $1 < \gamma < \frac{2q}{2-q}$. 
\end{theorem}

\begin{theorem}[Theorem 1.4.7 \cite{KlimBook} p.\,17]\label{onefourseven}
If $f \in L^q(D)$, $1 < q \leq 2$, then $T_D(f)|_{\p D(0,r)} \in L^\gamma(\p D(0,r))$, $0 < r \leq 1$, where $\gamma$ satisfies $1 < \gamma < \frac{q}{2-q}$, and 
\[
||T_Df||_{L^\gamma(\p D(0,r))} \leq C ||f||_{L^q(D)},
\]
where $C$ is a constant that does not depend on $r$ or $f$. 
\end{theorem}

\begin{theorem}[Theorem 1.4.8 \cite{KlimBook} p.\,18]\label{onefoureight}
If $f \in L^q(D)$, $1 < q \leq 2$, then 
\[
\lim_{r\nearrow 1} \int_0^{2\pi} |T_D(f)(e^{i\theta}) - T_D(f)(re^{i\theta})|^\gamma \,d\theta = 0, 
\]
for $1 \leq \gamma < \frac{q}{2-q}$. 
\end{theorem}

\begin{deff}
For $f: \mathbb{C} \to \mathbb{C}$ and $z  \in \mathbb{C}$, we denote by $\widetilde{T}(\cdot)$ the integral operator defined by 
\[
\widetilde{T}(f)(z) = -\frac{1}{\pi} \iint_D \left(\frac{f(\zeta)}{\zeta - z} + \frac{z\overline{f(\zeta)}}{1-\overline{\zeta}z} \right)\, d\xi\,d\eta,
\]
whenever the integral is defined.
\end{deff}

Similarly to $T_D(\cdot)$, $\widetilde{T}(f)$ is defined for any $f \in L^1(D)$, see Theorem 1.4.2 of \cite{KlimBook} (where the result is cited from \cite{Vek}). Also from p.16 of \cite{KlimBook}, the operator $\widetilde{T}(\cdot)$ shares the property of being a right inverse to the Cauchy-Riemann operator, which is well-known for the operator $T_D(\cdot)$ (see \cite{Vek}), i.e., 
\[
\frac{\p}{\p\z} \widetilde{T}(f) = f,
\]
whenever $\widetilde{T}(f)$ is defined. Note that Theorems \ref{Vekonepointtwentysix} and \ref{onefourseven} also hold for $\widetilde{T}(\cdot)$, see their cited references. Observe that
\[
\frac{f(\zeta)}{\zeta - z} + \frac{z\overline{f(\zeta)}}{1-\overline{\zeta}z} = \frac{f(\zeta)}{\zeta - z} - \overline{\left( \frac{f(\zeta)}{\zeta - z} \right)} = 2i \im\left\{\frac{f(\zeta)}{\zeta - z} \right\}
\]
for $|z| = 1$, so 
\[
\re\{(\widetilde{T}(f))_b\} = 0.
\]

We will need one more integral operator before proceeding.

\begin{deff}
For $f: \mathbb{C} \to \mathbb{C}$ and $z \in \mathbb{C}$, we denote by $\mathcal{T}_D(\cdot)$ the integral operator defined by 
\[
\mathcal{T}_D(f)(z) = -\frac{1}{2\pi} \iint_D \left(\frac{f(\zeta)}{\zeta} \frac{\zeta+z}{\zeta - z} + \frac{\overline{f(\zeta)}}{\overline{\zeta}} \frac{1+z\overline{\zeta}}{1-z\overline{\zeta}} \right)\, d\xi\,d\eta,
\]
whenever the integral is defined.
\end{deff}

From \cite{BegBook}, $\mathcal{T}_D$ is well-defined for all $f \in L^1(D)$. By Theorem 2.1 in \cite{Beg}, the Schwarz boundary value problem 
\begin{numcases}{}
\frac{\p w}{\p\z} = f \nonumber \\
\re{w|_{\p D}} = h \nonumber\\
\im{w(0)} = c \nonumber,
\end{numcases}
for $f \in L^1(D)$, real-valued $h \in C(\p D)$, and $c \in \mathbb{R}$, is uniquely solved by the function 
\begin{align*}
w(z) &= ic + \frac{1}{2\pi i} \int_{|\zeta| = 1} h(\zeta) \,\frac{\zeta + z}{\zeta - z}\,\frac{d\zeta}{\zeta} - \frac{1}{2\pi}\iint_{|\zeta|<1} \left( \frac{f(\zeta)}{\zeta}\,\frac{\zeta + z}{\zeta - z} + \frac{\overline{f(\zeta)}}{\overline{\zeta}}\,\frac{1+z\overline{\zeta}}{1-z\overline{\zeta}}\right)\,d\xi\,d\eta \\
&=  ic + \frac{1}{2\pi i} \int_{|\zeta| = 1} h(\zeta) \,\frac{\zeta + z}{\zeta - z}\,\frac{d\zeta}{\zeta} + \mathcal{T}_D(f)(z),
\end{align*}
because $\mathcal{T}_D(\cdot)$ is also a right-inverse to $\frac{\p}{\p\z}$ like $T_D(\cdot)$ and $\widetilde{T}(\cdot)$, has zero real part on the boundary like $\widetilde{T}(\cdot)$, and has zero imaginary part when $z=0$, which can be seen by direct evaluation. These claims are clear by rewriting the operator in the following way (see \cite{hoio}):
\begin{align*}
\mathcal{T}_D(f)(z) &= -\frac{1}{2\pi} \iint_D \left(\frac{f(\zeta)}{\zeta} \frac{\zeta+z}{\zeta - z} + \frac{\overline{f(\zeta)}}{\overline{\zeta}} \frac{1+z\overline{\zeta}}{1-z\overline{\zeta}} \right)\, d\xi\,d\eta \\
&= -\frac{1}{\pi} \iint_D \left(\frac{f(\zeta)}{\zeta - z} + \frac{z\overline{f(\zeta)}}{1-\overline{\zeta}z} \right)\, d\xi\,d\eta + \frac{1}{2\pi} \iint_D \frac{f(\zeta)}{\zeta}\,d\xi\,d\eta - \frac{1}{2\pi} \iint_D \frac{\overline{f(\zeta)}}{\overline{\zeta}}\,d\xi\,d\eta \\
&= \widetilde{T}(f)(z)+ \frac{1}{2\pi} \iint_D \frac{f(\zeta)}{\zeta}\,d\xi\,d\eta - \frac{1}{2\pi} \iint_D \frac{\overline{f(\zeta)}}{\overline{\zeta}}\,d\xi\,d\eta .
\end{align*}
Note that the last two integrals are simply numbers. Hence, the analogous statements of Theorems \ref{Vekonepointtwentysix} and \ref{onefourseven} also hold for this operator.

Higher-order variations of the Schwarz boundary value problem are solvable by utilizing iterations of the operator $\mathcal{T}_D(\cdot)$. From \cite{Beg}, we have the following.

\begin{prop}\label{iteratedT}
For $f \in L^1(D)$ and $n \in \mathbb{N}$,
\[
\mathcal{T}_D^n(f)(z) = - \frac{1}{2\pi}\iint_{|\zeta|<1} \left( \frac{f(\zeta)}{\zeta}\,\frac{\zeta + z}{\zeta - z} + \frac{\overline{f(\zeta)}}{\overline{\zeta}}\,\frac{1+z\overline{\zeta}}{1-z\overline{\zeta}}\right)(\zeta - z+\overline{\zeta - z})^{n-1}\,d\xi\,d\eta,
\]
where $\mathcal{T}_D^n(f)$ is the operator $\mathcal{T}_D$ applied  $n$-times to the function $f$.
\end{prop}

\subsection{Boundary values in the sense of distributions}

 First, we define a boundary value in the sense of distributions as in \cite{GHJH2} and \cite{WB}. 

\begin{deff}\label{bvcircle}Let $f$ be a function defined on $D$. We say that $f$ has a boundary value in the sense of distributions, denoted by $f_b \in \mathcal{D}'(\p D)$, if, for every $ \varphi \in C^\infty(\partial D)$, the limit
            \[
             \langle f_b, \varphi \rangle := \lim_{r \nearrow 1} \int_0^{2\pi} f(re^{i\theta}) \, \varphi(\theta) \,d\theta
            \]
            exists.
            
\end{deff}

We now define the space of holomorphic functions associated with this kind of boundary values.

\begin{deff}
We define $H_b$ to be that subset of functions in $H(D)$ that have boundary values in the sense of distributions.
\end{deff}

From Theorem 2.2, Theorem 3.1, and Corollary 3.1 of \cite{GHJH2}, we know that the functions $w \in \Har$ have boundary values in the sense of distributions $w_b$, and that $\cup_{0<p\leq 1} H^p(D)$ is a proper subset of $H_b$, see \cite{Duren} for an example of a holomorphic function of tempered growth that has nontangential boundary values almost nowhere.

\subsection{Schwarz Boundary Value Problem}\label{SchwarzSection}

We prove the following theorem which extends the result of \cite{Beg} to a more general boundary condition.

\begin{theorem}
The Schwarz boundary value problem 
\begin{numcases}{}
\frac{\p w}{\p\z} = f \nonumber \\
\re{w_b} = \re{h_b} \nonumber\\
\im{w(0)} = c, \nonumber
\end{numcases}
for $f \in L^1(D)$, $h \in  H_b$, and $c \in \mathbb{R}$, is solved by the function 
\[
w(z) = ic  - I + \frac{1}{2\pi } \langle h_b, P_r(\theta - \cdot)\rangle - \frac{1}{2\pi}\iint_{|\zeta|<1} \left( \frac{f(\zeta)}{\zeta}\,\frac{\zeta + z}{\zeta - z} + \frac{\overline{f(\zeta)}}{\overline{\zeta}}\,\frac{1+z\overline{\zeta}}{1-z\overline{\zeta}}\right)\,d\xi\,d\eta,
\]
where 
\[
P_r(\theta) = \frac{1-r^2}{1-2r\cos(\theta)+r^2}, 
\]
and
\[
I:= \frac{i}{2\pi}\langle \im h_b, 1\rangle .
\]
\end{theorem}

\begin{proof}
Since $h \in H_b$, it follows from Theorem 3.1 of \cite{GHJH2} that 
\[
h(re^{i\theta}) = \frac{1}{2\pi}\langle h_b, P_r(\theta - \cdot) \rangle.
\]
From \cite{Beg} and \cite{BegBook}, we know that $\mathcal{T}_D(\cdot)$ is a right-inverse operator of $\frac{\p}{\p\z}$. So, , 
\begin{align*}
\frac{\p}{\p\z}\left( ic - I + \frac{1}{2\pi } \langle h_b, P_r(\theta - \cdot)\rangle - \frac{1}{2\pi}\iint_{|\zeta|<1} \left( \frac{f(\zeta)}{\zeta}\,\frac{\zeta + z}{\zeta - z} + \frac{\overline{f(\zeta)}}{\overline{\zeta}}\,\frac{1+z\overline{\zeta}}{1-z\overline{\zeta}}\right)\,d\xi\,d\eta\right) = f.
\end{align*}
Since 
\[
h_b = \left(\frac{1}{2\pi}\langle h_b, P_r(\theta - \cdot) \rangle\right)_b
\]
 and by Theorem \ref{onefoureight}
 \[
T(f)(z)|_{\p D} = T(f)_b
 \]
as distributions, it follows that, for $|z| = 1$,  
\begin{align*}
w_b &= ic - I + h_b  - \frac{1}{2\pi}\iint_{|\zeta|<1} \left( \frac{f(\zeta)}{\zeta}\,\frac{\zeta + z}{\zeta - z} + \frac{\overline{f(\zeta)}}{\overline{\zeta}}\,\frac{1+z\overline{\zeta}}{1-z\overline{\zeta}}\right)\,d\xi\,d\eta \\
&= ic - I + h_b  - \frac{1}{2\pi}\iint_{|\zeta|<1} \left( \frac{f(\zeta)}{\zeta}\,\frac{\zeta + z}{\zeta - z} - \frac{\overline{f(\zeta)}}{\overline{\zeta}}\,\frac{\overline{\zeta + z} }{\overline{\zeta -z}}\right)\,d\xi\,d\eta \\
&= ic - I + h_b  - \frac{i}{\pi}\iint_{|\zeta|<1} \im\left\{ \frac{f(\zeta)}{\zeta}\,\frac{\zeta + z}{\zeta - z} \right\}\,d\xi\,d\eta .
\end{align*}
Since $ic$, $I$, and the integral are purely imaginary, it follows that 
\[
\re\{w_b\} = \re\{h_b\}.
\]

Now, evaluating when $z = 0 $, we have
\begin{align*}
w(0) &= ic - I + \frac{1}{2\pi } \langle h_b, P_0(\theta - \cdot)\rangle - \frac{1}{2\pi}\iint_{|\zeta|<1} \left( \frac{f(\zeta)}{\zeta}\,\frac{\zeta + 0}{\zeta - 0} + \frac{\overline{f(\zeta)}}{\overline{\zeta}}\,\frac{1+0\cdot\overline{\zeta}}{1-0\cdot\overline{\zeta}}\right)\,d\xi\,d\eta \\
&= ic - I + \frac{1}{2\pi } \langle h_b, 1\rangle - \frac{1}{2\pi}\iint_{|\zeta|<1} \left( \frac{f(\zeta)}{\zeta} + \frac{\overline{f(\zeta)}}{\overline{\zeta}}\right)\,d\xi\,d\eta \\
&= ic - I + \frac{1}{2\pi } \langle h_b, 1\rangle - \frac{1}{\pi}\iint_{|\zeta|<1} \re\left\{ \frac{f(\zeta)}{\zeta} \right\}\,d\xi\,d\eta \\
&= ic + \frac{1}{2\pi}\langle \re h_b, 1 \rangle - \frac{1}{\pi}\iint_{|\zeta|<1} \re\left\{ \frac{f(\zeta)}{\zeta} \right\}\,d\xi\,d\eta .
\end{align*}
Thus, 
\[
\im\{w(0)\} = c.
\]
\end{proof}

\begin{comm}
If, in the above theorem, $f \in L^q(D)$, $q>1$, and $h \in H^p(D)$, $0 < p \leq 1$, then $w$ is an element of the class $H^p_{f,q}(D)$ studied in \cite{WB}. 
\end{comm}

\subsection{Extension to higher-order equations}

By Theorem 3.4 in \cite{Beg}, the higher-order variant of the Schwarz boundary value problem 
\begin{numcases}{}
\frac{\p^n w}{\p\z^n} = f, \nonumber\\
\re\left\{\frac{\p^k w}{\p\z^k}|_{\p D}\right\} = h_k, \nonumber\\
\im\left\{\frac{\p^k w}{\p\z^k}(0)\right\} = c_k, \nonumber
\end{numcases}
for $n \in \mathbb{N}$, $f \in L^1(D)$, real-valued $h_k \in C(\p D)$, and $c_k \in \mathbb{R}$, for $k = 0, 1, \ldots, n-1$, is uniquely solved by the function 
\begin{align*}
w(z) &= i\sum_{k = 0}^{n-1}\frac{c_k}{k!}(z+\z)^k + \sum_{k=0}^{n-1}\frac{(-1)^k}{2\pi i k!}\int_{|\zeta|=1} h_k(\zeta) \,\frac{\zeta + z}{\zeta - z}\,(\zeta - z+ \overline{\zeta - z})^k\,\frac{d\zeta}{\zeta} \\
&\quad - \frac{1}{2\pi}\iint_{|\zeta|<1} \left( \frac{f(\zeta)}{\zeta}\,\frac{\zeta + z}{\zeta - z} + \frac{\overline{f(\zeta)}}{\overline{\zeta}}\,\frac{1+z\overline{\zeta}}{1-z\overline{\zeta}}\right)(\zeta - z+\overline{\zeta - z})^{n-1}\,d\xi\,d\eta.
\end{align*}

We define a space that we will use in the theorem that follows. 

\begin{deff}
For each $0 < p \leq 1$, define $\re\{(H^p(D))_b\}$ as the set of real parts of the boundary values in the sense of distributions of the functions in $H^p(D)$. 
\end{deff}

We prove the following theorem which extends the problem to more general boundary conditions.

\begin{theorem}\label{higherschwarz}
The Schwarz boundary value problem 
\begin{numcases}{}
\frac{\p^n w}{\p\z^n} = f \nonumber \\
\re\left\{\left(\frac{\p^k w}{\p\z^k}\right)_b\right\} = h_k , \quad 1 \leq k \leq n -1 \nonumber\\
\re\{w_b\} = \re\{(h_0)_b\}  \nonumber\\
\im\left\{\frac{\p^k w}{\p\z^k}(0)\right\} = c_k, \quad 0 \leq k \leq n -1, \nonumber
\end{numcases}
for $n \in \mathbb{N}$, $f \in L^1(D)$, $h_k \in \re\{(H^{p_k}(D))_b\}$ where $\frac{1}{2} < p_k \leq 1$ for $k = 1, 2, \ldots, n-1$, $h_0 \in H_b$, and $c_k \in \mathbb{R}$ for $k = 0, 1, 2,\ldots, n-1$, is solved by the function 
\begin{align*}
w(z) &= i\sum_{k = 1}^{n-1}\frac{c_k}{k!}(z+\z)^k  - I + \frac{1}{2\pi} \langle (h_0)_b, P_r(\theta - \cdot)\rangle \\
&\quad+ \sum_{k=1}^{n-1}\frac{(-1)^k}{2\pi   k!} \langle h_k, (P_r(\theta - \cdot)+iQ_r(\theta - \cdot))(e^{i(\cdot)} - re^{i\theta}+\overline{e^{i(\cdot)} - re^{i\theta}})^k\rangle \\
&\quad - \frac{1}{2\pi}\iint_{|\zeta|<1} \left( \frac{f(\zeta)}{\zeta}\,\frac{\zeta + z}{\zeta - z} + \frac{\overline{f(\zeta)}}{\overline{\zeta}}\,\frac{1+z\overline{\zeta}}{1-z\overline{\zeta}}\right)(\zeta - z+\overline{\zeta - z})^{n-1}\,d\xi\,d\eta,
\end{align*}
where 
\[
I = \frac{1}{2\pi}\langle \im\{(h_0)_b\}, 1\rangle,
\]
\[
P_r(\theta) = \frac{1-r^2}{1-2r\cos(\theta)+r^2}, 
\]
\[
Q_r(\theta) = \frac{2r\sin(\theta)}{1-2r\cos(\theta) + r^2},
\]
and $z = re^{i\theta}$.
\end{theorem}

\begin{proof}
The case $n = 1$ is the last theorem. 

The cases $n > 1$ are proved by observing that the formula is the result of iteration of the $n = 1$ result by the same method as the proof of Theorem 3.4 in \cite{Beg}. 

Our choice of hypothesis is key here. By choosing $h_k \in \re\{(H^{p_k}(D))_b\}$ with $\frac{1}{2} < p_k \leq 1$ here, we know that 
\[
\langle h_k, P_r(\theta - \cdot)+iQ_r(\theta - \cdot)\rangle
\]
is not only an element of $H^{p_k}(D)$, see Theorem 6.2 in \cite{GHJH2}, but the function is in $L^1(D)$, see Lemma 1.8.3 in \cite{KlimBook}. So, the iteration is well-defined, and the result of that iteration is 
\[
\mathcal{T}_D^k(\langle h_k, P_r(\theta - \cdot)+iQ_r(\theta - \cdot)\rangle) = \langle h_k, (P_r(\theta - \cdot)+iQ_r(\theta - \cdot))(e^{i(\cdot)} - re^{i\theta}+\overline{e^{i(\cdot)} - re^{i\theta}})^k\rangle,
\]
 by application of Fubini's theorem. 
 
\end{proof}

\begin{comm}
If, in the last theorem, $f \in L^q(D)$, $q > 1$ and $h_0 \in H^{p_0}(D)$, $0 < p_0 \leq 1$, then the solution $w$ is an element of the class $H^{n,p}_{f,q}(D)$, where $p:= \min_{0 \leq k \leq n-1}\{p_k\}$, that is studied in \cite{WB}.
\end{comm}

\subsection{A special case} Next, we show a certain special case of the first order Schwarz boundary value problem is solvable by appealing to the higher-order extension above. This special case is notable as the associated differential equation has nonhomogeneous part which is not assumed to be in any Lebesgue space. First, we recall some well known representation results. 

\begin{prop}[\cite{Balk} or Theorem 2.1 \cite{polyhardy}]\label{Balkdecomp}
For $n$ a positive integer, the functions $f$ that solve 
\[
\frac{\p^n f}{\p\z^n} = 0
\]
have the form 
\[
f(z) = \sum_{k=0}^{n-1} \z^k f_k(z),
\]
where each $f_k \in H(D)$. 
\end{prop}

\begin{prop}[Theorem 4.9 \cite{l2poly}]\label{altrepnohardy}
Every $w $ that solves
\[
\frac{\p^n w}{\p\z^n} = f,
\] 
where $f $ solves
\[
\frac{\p^n f}{\p\z^n} = 0
\]
and has representation $f = \sum_{k=0}^{n-1} \z^k f_k(z)$ from Proposition \ref{Balkdecomp}, has the representation
\[
w(z) = w_0 - \sum_{k=1}^n \frac{(-1)^k}{k!}\z^k \frac{\p^{k-1}f}{\p\z^{k-1}}(z),
\]
where $w_0 \in H(D)$.
\end{prop}

With these representation formulas in hand, we are ready to take advantage of the additional structure provided a solution of a nonhomogeneous Cauchy-Riemann equation when the nonhomogeneous part itself solves a homogeneous $n$-th order homogeneous Cauchy-Riemann equation. In particular, we solve a first-order Schwarz boundary value problem where the nonhomogeneous part of the corresponding differential equation is not necessarily integrable.

\begin{theorem}\label{schwarznoint}
For every $f$ that solves the Schwarz boundary value problem
\begin{numcases}{}
\frac{\p^n f}{\p\z^n} = 0 \nonumber \\
\re\left\{\left(\frac{\p^k f}{\p\z^k}\right)_b\right\} = h_k , \quad 1 \leq k \leq n -1 \nonumber\\
\re\{f_b\} = \re\{(h_0)_b\} \nonumber\\
\im\left\{\frac{\p^k f}{\p\z^k}(0)\right\} = c_k, \quad 0 \leq k \leq n -1, \nonumber
\end{numcases}
for $n \in \mathbb{N}$, $h_k \in \re\{(H^{p_k}(D))_b\}$ where $\frac{1}{2} < p_k \leq 1$ for $k = 1, 2, \ldots, n-1$, $h_0 \in H_b$, and $c_k \in \mathbb{R}$ for $k = 0, 1, 2,\ldots, n-1$, the first-order Schwarz boundary value problem
\begin{numcases}{}
\frac{\p w}{\p\z} = f  \nonumber\\
\re\{w_b\} = \re\left\{h_b  - \sum_{k=1}^n \frac{(-1)^k}{k!}e^{-ik(\cdot)} \left(\frac{\p^{k-1}f}{\p\z^{k-1}} \right)_b \right\} \nonumber\\
\im\{w(0)\} = c, \nonumber
\end{numcases}
with $h \in  H_b$, and $c \in \mathbb{R}$, is solved by the function
\begin{align*}
w(z) 
& = ic  - I + \frac{1}{2\pi } \langle h_b, P_r(\theta - \cdot)\rangle - \z\left( - \widetilde{I} + \frac{1}{2\pi} \langle (h_0)_b, P_r(\theta - \cdot)\rangle\right) \\
&\quad - i\sum_{k=1}^n\sum_{\ell = k-1}^{n-1} \frac{(-1)^k c_\ell}{k!(\ell - k+1)!}\z^k(z+\z)^{\ell - (k-1)}  \\
&\quad  - \sum_{k=1}^n \sum_{\ell=k-1}^{n-1}\frac{(-1)^{\ell+1}}{2\pi k!(\ell-k+1)!} \langle h_\ell, (P_r(\theta - \cdot)+iQ_r(\theta - \cdot))(e^{i(\cdot)} - re^{i\theta}+\overline{e^{i(\cdot)} - re^{i\theta}})^{\ell-(k-1)}\rangle ,  
\end{align*}
where 
\[
I:= \frac{i}{2\pi}\langle \im h_b, 1\rangle 
\]
and 
\[
\widetilde{I} := \frac{i}{2\pi}\langle \im (h_0)_b, 1\rangle.
\]

\end{theorem}

\begin{proof}
The existence of an $f$ that satisfies the hypothesis here is the result of Theorem \ref{higherschwarz}, and from that theorem, we know it has the form 
\begin{align*}
f(z) &=  i\sum_{\ell = 1}^{n-1}\frac{c_\ell}{\ell!}(z+\z)^\ell  - \widetilde{I} + \frac{1}{2\pi} \langle (h_0)_b, P_r(\theta - \cdot)\rangle \\
&\quad+ \sum_{j=1}^{n-1}\frac{(-1)^j}{2\pi   j!} \langle h_j, (P_r(\theta - \cdot)+iQ_r(\theta - \cdot))(e^{i(\cdot)} - re^{i\theta}+\overline{e^{i(\cdot)} - re^{i\theta}})^j\rangle .
\end{align*}
By Proposition \ref{altrepnohardy}, we know that a $w$ which solves 
\[
\frac{\p w}{\p\z} = f
\]
has the form 
\begin{equation}\label{wrepwithderivs}
w(z) = w_0 - \sum_{k=1}^n \frac{(-1)^k}{k!}\z^k \frac{\p^{k-1}f}{\p\z^{k-1}}(z),
\end{equation}
where $w_0 \in H(D)$. Since we have an explicit representation for $f$ here, we can find its partial derivatives by computation. Observe that 
\begin{align*}
\frac{\p^k f}{\p\z^k} 
&= i\sum_{\ell = k}^{n-1} \frac{c_\ell}{(\ell-k)!}(z+\z)^{\ell-k}  + F(re^{i\theta},\ell) \\
&\quad+ \sum_{j=k}^{n-1}\frac{(-1)^j}{2\pi   (j-k)!} \langle h_j, (P_r(\theta - \cdot)+iQ_r(\theta - \cdot))(e^{i(\cdot)} - re^{i\theta}+\overline{e^{i(\cdot)} - re^{i\theta}})^{j-k}\rangle,
\end{align*}
where
   \begin{numcases}{F(re^{i\theta}, \ell) =  }
             - \widetilde{I} + \frac{1}{2\pi} \langle (h_0)_b, P_r(\theta - \cdot)\rangle, & $\ell = 0$, \nonumber\\
             0, & otherwise \nonumber.
    \end{numcases}
After substituting into equation (\ref{wrepwithderivs}), $w$ has the form 
\begin{align*}
&w(z)  \\
& = w_0(z)
- \sum_{k=1}^n \frac{(-1)^k}{k!} \z^k \left[ i\sum_{\ell = k-1}^{n-1}\frac{c_\ell}{(\ell - (k-1))!}(z+\z)^{\ell - (k-1)}  + F(re^{i\theta}, \ell) \right] \\
&\quad  - \sum_{k=1}^n \frac{(-1)^k}{k!} \z^k\left[\sum_{\ell = k - 1}^{n-1}\frac{(-1)^{\ell-(k-1)}}{2\pi (\ell-(k-1))!} \langle h_\ell, (P_r(\theta - \cdot)+iQ_r(\theta - \cdot))(e^{i(\cdot)} - re^{i\theta}+\overline{e^{i(\cdot)} - re^{i\theta}})^{\ell-(k-1)}\rangle \right],
\end{align*}
where $w_0$ is an arbitrary holomorphic function. We choose $w_0$ to be the function 
\[
w_0(re^{i\theta}) := ic  - I + \frac{1}{2\pi } \langle h_b, P_r(\theta - \cdot)\rangle . 
\]
Now, $w$ has the form
\begin{align*}
&w(z) \\
& = ic  - I + \frac{1}{2\pi } \langle h_b, P_r(\theta - \cdot)\rangle - \z\left( - \widetilde{I} + \frac{1}{2\pi} \langle (h_0)_b, P_r(\theta - \cdot)\rangle\right) \\
&\quad - \sum_{k=1}^n \frac{(-1)^k}{k!} \z^k \left[ i\sum_{\ell = k-1}^{n-1}\frac{c_\ell}{(\ell - (k-1))!}(z+\z)^{\ell - (k-1)}  \right] \\
&\quad  - \sum_{k=1}^n \frac{(-1)^k}{k!} \z^k\left[\sum_{\ell = k - 1}^{n-1}\frac{(-1)^{\ell-(k-1)}}{2\pi (\ell-(k-1))!} \langle h_\ell, (P_r(\theta - \cdot)+iQ_r(\theta - \cdot))(e^{i(\cdot)} - re^{i\theta}+\overline{e^{i(\cdot)} - re^{i\theta}})^{\ell-(k-1)}\rangle \right].
\end{align*}
By construction, this $w$ solves 
\[
\frac{\p w}{\p\z} = f.
\]
Now, observe that 
\begin{align*}
\re\{w_b\} &= \re\left\{ \left( ic  - I + \frac{1}{2\pi } \langle h_b, P_r(\theta - \cdot)\rangle  - \sum_{k=1}^n \frac{(-1)^k}{k!}\z^k \frac{\p^{k-1}f}{\p\z^{k-1}} \right) \right\} \\
&= \re\left\{ h_b  - \sum_{k=1}^n \frac{(-1)^k}{k!}e^{-ik(\cdot)} \left(\frac{\p^{k-1}f}{\p\z^{k-1}} \right)_b \right\}.
\end{align*}
Finally, observe that 
\begin{align*}
\im\{w(0)\} = \im\{ic - \frac{i}{2\pi}\langle \im h_b, 1\rangle + \frac{1}{2\pi}\langle  h_b, 1\rangle\} = c.
\end{align*}
\end{proof}

\begin{corr}
The constructed solution $w$ from Theorem \ref{schwarznoint} also solves the higher-order Schwarz boundary value problem
\begin{numcases}{}
\frac{\p^{n+1} w}{\p\z^{n+1}} = 0 \nonumber \\
\re\left\{\left(\frac{\p^{k+1} w}{\p\z^{k+1}}\right)_b\right\} = h_{k} , \quad 1 \leq k \leq n -1 \nonumber \\
\re\left\{\left(\frac{\p w}{\p\z}\right)_b\right\} = \re\{(h_0)_b\}  \nonumber \\
\re\{w_b\} = \re\left\{h_b  - \sum_{k=1}^n \frac{(-1)^k}{k!}e^{-ik(\cdot)} \left(\frac{\p^{k-1}f}{\p\z^{k-1}} \right)_b \right\} \nonumber \\
\im\left\{\frac{\p^{k+1} w}{\p\z^{k+1}}(0)\right\} = c_k, \quad 0 \leq k \leq n -1 , \nonumber \\
\im\{w(0)\} = c. \nonumber
\end{numcases}
\end{corr}

\begin{comm}
While the corollary is a Schwarz boundary value problem corresponding to a homogeneous Cauchy-Riemann equation (and 0 is certainly integrable), note that the boundary condition for the first derivative is in terms of a boundary value in the sense of distributions of a holomorphic function which is not necessarily a member of a $H^p(D)$ space. This situation has not been handled by any of the other solved cases in the literature. On the other hand, if in the hypothesis of Theorem \ref{schwarznoint} we choose $(h_0)_b$ to be a boundary value in the sense of distributions of a function in a holomorphic Hardy space $H^{p_0}(D)$, $0 < p_0 \leq 1$, then $f$ is a member of the space $H^{n,p}_0(D)$, where $p:=\min_{0\leq k \leq n-1}\{p_k\}$, which are studied in \cite{WB}. If we choose $h_b$ to be the boundary value in the sense of distributions of a function in a holomorphic Hardy space $H^q(D)$, $0 < q \leq 1$, then $w$ is an element of the spaces $H^{1,\tilde{p}}_f(D)$ and $H^{n+1,\tilde{p}}_0(D)$, where $\tilde{p}:= \min\{q,p\}$, see \cite{WB} for more details about these spaces. 
\end{comm}

We now further generalize Theorem \ref{schwarznoint}. By iterating the construction used in the last theorem, we construct solutions to a Schwarz problem that does not appeal to Hardy space theory for any order.  

\begin{theorem}
For $n$ a positive integer, the Schwarz problem 
    \begin{numcases}{}\label{hos}
        \frac{\p^n f_n}{\p\z^n} = 0 \nonumber\\
        \re\left\{ \left( \frac{\p^k f_n}{\p\z^k}\right)_b \right\} = \re\left\{ (f_{n-1-k})_b\right\} =  \re\left\{ (h_{n-k})_b - \sum_{\ell = 1}^{n-k-1} \frac{(-1)^\ell}{\ell!} e^{-i\ell(\cdot)} (f_{n-\ell})_b\right\} \\
        \im\left\{ \frac{\p^k f_n}{\p\z^k}(0) \right\} = c_{n-1-k} \nonumber
    \end{numcases}
where each $f_k$ solves the Schwarz problem
    \begin{numcases}{}\label{fos}
        \frac{\p f_k}{\p\z} = f_{k-1} \nonumber\\
        \re\left\{ (f_k)_b\right\} = \re\left\{ (h_{k-1})_b - \sum_{\ell = 1}^{k-1} \frac{(-1)^\ell}{\ell!} e^{-i\ell(\cdot)} (f_{k-\ell})_b \right\} \\
        \im\left\{ f_k(0)\right\} =  c_{k-1} \nonumber
    \end{numcases}
with $h_{k-1} \in H_b$, and $c_{k-1} \in \R$,
for each $k$ with $1 \leq k \leq n$, is solved by 
\begin{align*}
 f_n(z) = ic_{n-1} - I_{n-1} + \frac{1}{2\pi} \langle (h_{n-1})_b, P_r(\theta - \cdot) \rangle - \sum_{\ell = 1}^{n-1} \frac{(-1)^\ell}{\ell!} \z^\ell f_{n-\ell}(z), 
\end{align*}
where each $f_k$ is described by 
\[
f_k(z) = ic_{k-1} - I_{k-1} + \frac{1}{2\pi} \langle (h_{k-1})_b, P_r(\theta - \cdot) \rangle - \sum_{\ell = 1}^{k-1} \frac{(-1)^\ell}{\ell!} \z^\ell f_{k-\ell}(z),
\]
with
\[
I_{k-1} := \frac{i}{2\pi} \langle \im\{(h_{k-1})_b\}, 1\rangle,
\] 
for $1 \leq k \leq n$, and $f_0 \equiv 0$.
\end{theorem}

\begin{proof}
We work by induction. 

For $n = 1$, observe that 
\[
f_1(z) = ic_0 - I_0 + \frac{1}{2\pi} \langle (h_0)_b, P_r(\theta - \cdot) \rangle
\]
is holomorphic, by Theorem 3.1 of \cite{GHJH2}, 
\[
\re\{(f_1)_b\} = \re\{ic_0 - I_0 + (h_0)_b\} = \re\{(h_0)_b\},
\]
and 
\[
\im{f_1(0)} = \im\left\{ic_0 - \frac{i}{2\pi} \langle \im\{(h_0)_b\}, 1\rangle + \frac{1}{2\pi} \langle (h_0)_b,1 \rangle\right\} = c_0.
\]

Suppose the theorem holds for all $n$ such that $1\leq n \leq m-1$. Consider $f_m$ defined by 
\[
f_m(z) = ic_{m-1} - I_{m-1} + \frac{1}{2\pi} \langle (h_{m-1})_b, P_r(\theta - \cdot) \rangle - \sum_{\ell = 1}^{m-1} \frac{(-1)^\ell}{\ell!} \z^\ell f_{m-\ell}(z).
\]
Observe that 
\begin{align*}
\frac{\p f_m}{\p\z} &= \frac{\p}{\p\z}\left(ic_{m-1} - I_{m-1} + \frac{1}{2\pi} \langle (h_{m-1})_b, P_r(\theta - \cdot) \rangle - \sum_{\ell = 1}^{m-1} \frac{(-1)^\ell}{\ell!} \z^\ell f_{m-\ell}\right) \\
&= - \sum_{\ell = 1}^{m-1} \frac{(-1)^\ell}{\ell!} \left[\ell\z^{\ell-1} f_{m-\ell} + \z^\ell f_{m-\ell - 1}\right],
\end{align*}
as 
\[
\frac{\p f_{k}}{\p\z} = f_{k-1},
\]
for all $1\leq k \leq m-1$. So, 
\begin{align*}
 &- \sum_{\ell = 1}^{m-1} \frac{(-1)^\ell}{\ell!} \left[\ell\z^{\ell-1} f_{m-\ell} + \z^\ell f_{m-\ell - 1}\right] \\
 &=  - \left(\sum_{\ell = 1}^{m-1} \frac{(-1)^\ell}{\ell!} \ell\z^{\ell-1} f_{m-\ell} + \sum_{\ell = 1}^{m-1} \frac{(-1)^\ell}{\ell!}\z^\ell f_{m-\ell - 1}  \right)\\
 &= - \left(-f_{m-1} + \sum_{\ell = 2}^{m-1} \frac{(-1)^\ell}{(\ell-1)!} \z^{\ell-1} f_{m-\ell} + \sum_{\ell = 1}^{m-1} \frac{(-1)^\ell}{\ell!}\z^\ell f_{m-\ell - 1}  \right).
\end{align*}
Since 
\[
\frac{\p f_1}{\p\z} = f_0 \equiv 0
\]
it follows that 
\begin{align*}
& - \left(-f_{m-1} + \sum_{\ell = 2}^{m-1} \frac{(-1)^\ell}{(\ell-1)!} \z^{\ell-1} f_{m-\ell} + \sum_{\ell = 1}^{m-1} \frac{(-1)^\ell}{\ell!}\z^\ell f_{m-\ell - 1}  \right) \\
&=  - \left(-f_{m-1} + \sum_{\ell = 2}^{m-1} \frac{(-1)^\ell}{(\ell-1)!} \z^{\ell-1} f_{m-\ell} + \sum_{\ell = 1}^{m-2} \frac{(-1)^\ell}{\ell!}\z^\ell f_{m-\ell - 1}  \right)\\
&=  - \left(-f_{m-1} + \sum_{\ell = 1}^{m-2} \frac{(-1)^{\ell+1}}{\ell!} \z^{\ell} f_{m-(\ell+1)} + \sum_{\ell = 1}^{m-2} \frac{(-1)^\ell}{\ell!}\z^\ell f_{m-\ell - 1}  \right) \\
&= f_{m-1}.
\end{align*}
Observe that, by construction, we have
\begin{align*}
\re\{(f_m)_b\} &= \re\left\{( ic_{m-1} - I_{m-1} + \frac{1}{2\pi} \langle (h_{m-1})_b, P_r(\theta - \cdot) \rangle - \sum_{\ell = 1}^{m-1} \frac{(-1)^\ell}{\ell!} \z^\ell f_{m-\ell})_b\right\} \\
&= \re\left\{ (h_{m-1})_b - \sum_{\ell = 1}^{m-1} \frac{(-1)^\ell}{\ell!} e^{-i\ell(\cdot)} (f_{m-\ell})_b\right\}
\end{align*}
and 
\begin{align*}
\im\{f_m(0)\} &= \im\left\{ic_{m-1} -  \frac{1}{2\pi} \langle (\im\{h_{m-1})_b\}, 1 \rangle + \frac{1}{2\pi} \langle (h_{m-1})_b, 1 \rangle \right\} = c_{m-1}.
\end{align*}
Thus, $f_m$ satisfies the conditions of the Schwarz problem (\ref{fos}). 

Now, since
\[
\frac{\p^{m-1}f_{m-1}}{\p\z^{m-1}} = 0, 
\]
it follows that 
\[
\frac{\p^{m}f_{m}}{\p\z^{m}} = \frac{\p^{m-1}}{\p\z^{m-1}}\left(\frac{\p f_m}{\p\z}\right) = \frac{\p^{m-1}}{\p\z^{m-1}}\left(f_{m-1}\right) =  0.
\]
The induction hypothesis confirms the boundary conditions and the pointwise conditions for every $\frac{\p^k f_m}{\p\z^k}$ except for $k = 0$, which is the case we have already confirmed. Therefore, $f_m$ satifies the conditions of the Schwarz problem (\ref{hos}).

\end{proof}

\begin{comm}
This extension of Theorem \ref{schwarznoint} and its corollary is interesting because it gives a constructive solution of a homogeneous Schwarz boundary value problem of arbitrary order where every boundary condition is exclusively in terms of boundary values in the sense of distributions of holomorphic functions, with no a priori need for those holomorphic functions to be members of a holomorphic Hardy space. However, if the $h_k$ are chosen to be in Hardy spaces $H^{p_k}(D)$, then the solutions $f_n$ will be elements of the classes $H^{1,p}_{f_{n-1}}(D)$ and $H^{n,p}_0(D)$, where $p:= \min_{0\leq k\leq n-1}\{p_k\}$, see \cite{WB} for more details of these classes of functions. 
\end{comm}

\section{The case of the upper half-plane}

\subsection{Definitions and Results}

We will let $\Cplus$ be the half-plane of complex numbers with positive imaginary part, $\Cminus$ the half-plane of complex numbers with negative imaginary part, $H(\Cplus)$ be the space of holomorphic functions on $\Cplus$, $H^p(\Cplus)$ be the classical Hardy spaces of holomorphic functions on $\Cplus$, $C(\R)$ the space of continuous functions on $\R$, $C^\infty_c(\R)$ be the space of smooth functions on $\R$ with compact support, $C^{0,\alpha}(\C)$ the functions that are H\"older continuous with H\"older exponent $\alpha$ on $\C$, and $\mathcal{D}'(\R)$ be the space of distributions on $\R$, i.e., the bounded linear functionals on $C^\infty_c(\R)$ with pairing represented by $\langle g, \varphi \rangle$ for $g \in \mathcal{D}'(\R)$ and $\varphi \in C^\infty_c(\R)$.

Similarly to the disk case, we begin by introducing integral operators, a useful space, and associated results.

\begin{deff}
We define the space $L^{p,\nu}(\C)$ of functions $f: \C \to \C$ as those functions that satisfy $f \in L^p(D)$ and $f_\nu \in L^p(D)$, where 
\[
f_\nu(z) := \frac{1}{|z|^{\nu}}\,f\left(\frac{1}{z}\right).
\]
\end{deff}

\begin{deff}
We define the space $L^{p,\nu}(\Cplus)$ of functions $f: \Cplus \to \C$ as those functions that satisfy $f \in L^p(D\cap \Cplus)$ and $f_\nu \in L^p(D\cap\Cminus)$.
\end{deff}

\begin{deff}
For $f: \C \to \C$, we denote by $T_{\Cplus}$ the integral operator defined by
\[
T_{\Cplus}(f)(z) = -\frac{1}{\pi} \iint_{\Cplus} \frac{f(\zeta)}{\zeta - z}\, d\xi\,d\eta,
\]
whenever the integral is defined. 
\end{deff}

\begin{deff}
For $f: \C \to \C$, we denote by $T_{\C}$ the integral operator defined by
\[
T_{\C}(f)(z) = -\frac{1}{\pi} \iint_{\C} \frac{f(\zeta)}{\zeta - z}\, d\xi\,d\eta,
\]
whenever the integral is defined. 
\end{deff}

\begin{deff}
For $f: \C \to \C$, we denote by $\mathcal{T}_{\Cplus}$ the integral operator defined by
\[
\mathcal{T}_{\Cplus}(f)(z) = -\frac{1}{\pi} \iint_{\Cplus} \left( f(\zeta)\left( \frac{1}{\zeta - z} - \frac{\zeta}{\zeta^2 + 1}\right) - \overline{f(\zeta)}\left( \frac{1}{\overline{\zeta}-z} - \frac{\overline{\zeta}}{\overline{\zeta}^2 + 1}\right)\right) \,d\xi\,d\eta,
\]
whenever the integral is defined. 
\end{deff}

As in the case of $T_D$ and its associated operators, these integral operators are well-defined when acting on integrable functions and are right-inverse operators to $\frac{\p}{\p\z}$, see \cite{Vek}. 

\begin{theorem}[Theorem 1.23 \cite{Vek}]\label{regofT}
For any $f \in L^{p,2}(\C)$, $p > 2$, $T_{\C}(f) \in L^{p,2}(\C) \cap C^{0,\alpha}(\C)$, with $\alpha = \frac{p-2}{p}$.
\end{theorem}

By Theorem 17 in \cite{Gaertner}, the Schwarz boundary value problem 
\begin{numcases}{}
\frac{\p w}{\p\z} = f \nonumber\\
\re{w|_{\R}} = h \nonumber\\
\im{w(i)} = c , \nonumber
\end{numcases}
for $f \in L^{p,2}(\Cplus)$, real-valued and bounded $h \in C(\R)$, and $c \in \mathbb{R}$, is uniquely solved by the function 
\begin{align*}
w(z) &= ic + \frac{1}{\pi i} \int_{-\infty}^\infty h(t)\left( \frac{1}{t-z} - \frac{t}{t^2+1} \right)\,dt\\
&\quad\quad -\frac{1}{\pi} \iint_{\Cplus} \left( f(\zeta)\left( \frac{1}{\zeta - z} - \frac{\zeta}{\zeta^2 + 1}\right) - \overline{f(\zeta)}\left( \frac{1}{\overline{\zeta}-z} - \frac{\overline{\zeta}}{\overline{\zeta}^2 + 1}\right)\right) \,d\xi\,d\eta.
\end{align*}

Note that the first integral is a holomorphic function with real part that satisfies the boundary condition and with imaginary part that is zero when $z = i$, and the second integral is a particular solution to the nonhomogeneous Cauchy-Riemann equation with vanishing real part on the boundary and imaginary part that is zero when $z = i$, see \cite{Gaertner}. 

The operator $\mathcal{T}_{\Cplus}(\cdot)$ can be iterated when appropriate to solve higher-order nonhomoegeneous Cauchy-Riemann equations of the form $\frac{\p^n w}{\p\z^n} = f$ for $n$ a positive integer greater than one. We collect some results concerning $\mathcal{T}_{\Cplus}^n(\cdot)$ here. 

\begin{theorem}[Definition 2.1, Theorem 2.2  \cite{higherupper}]\label{higherupperth1}
For $f \in L^{p,2}(\Cplus)$, the operator $\mathcal{T}_{\Cplus}^n(f) := \underbrace{(\mathcal{T}_{\Cplus}\circ {\mathcal{T}_{\Cplus}}\circ \cdots \circ {\mathcal{T}_{\Cplus}})}_{n-\text{times}}(f) $ is well-defined and representable as 
\begin{align*}
&\mathcal{T}_{\Cplus}^n(f)(z) \\&= 
\frac{(-1)^n}{\pi (n-1)!} \iint_{\Cplus}(\zeta - z + \overline{\zeta - z})^{n-1}\left[ f(\zeta)\left(\frac{1}{\zeta - z} - \frac{\zeta}{\zeta^2+1}\right) - \overline{f(\zeta)}\left(\frac{1}{\overline{\zeta} - z} - \frac{\overline{\zeta}}{\overline{\zeta}^2+1}\right)\right]\,d\xi\,d\eta.
\end{align*}
\end{theorem}

\begin{theorem}[Theorem 2.4 \cite{higherupper}]\label{higherupperth2}
Let $f \in L^{p,2}(\Cplus)$ and $n$ be a positive integer. Then $\mathcal{T}_{\Cplus}^n(f)$ has the following properties:
\begin{itemize}
    \item $\frac{\p^\ell}{\p\z^\ell} \mathcal{T}_{\Cplus}^n(f) = \mathcal{T}_{\Cplus}^{n-\ell}(f)$, for $0\leq \ell \leq n$, 
    \item $\re\left\{\frac{\p^\ell}{\p\z^\ell} \mathcal{T}_{\Cplus}^n(f)\big|_{\mathbb{R}}\right\} = 0$, for $0\leq \ell \leq n-1$,
    \item $\im\left\{\frac{\p^\ell}{\p\z^\ell} \mathcal{T}_{\Cplus}^n(f)(i)\right\} = 0$, for $0\leq \ell \leq n-1$.
\end{itemize}
\end{theorem}

\subsection{Boundary values in the sense of distributions}

In the upper half-plane setting, we have the following analogue of Definition \ref{bvcircle}. 

\begin{deff}\label{bvline}
Let $f$ be a function defined on $\Cplus$. We say that $f$ has a boundary value in the sense of distributions, denoted by $f_b \in \mathcal{D}'(\R)$, if, for every $ \varphi \in C^\infty_c(\R)$, the limit
            \[
             \langle f_b, \varphi \rangle := \lim_{y\searrow 0} \int_{-\infty}^{\infty} f(x+iy) \, \varphi(x) \,dx
            \]
            exists.
            
\end{deff}

In the next section, we will need to consider the boundary values in the sense of distributions of functions in $H(\Cplus)$ and functions $T_{\Cplus}(f)$, for $f \in L^{p,2}(\Cplus)$. We show that for both of these kinds of functions the boundary value in the sense of distributions exists.

\begin{prop}
For any $f \in L^{p,2}(\Cplus)$, $T_{\Cplus}(f)$ has a boundary value in the sense of distributions. 
\end{prop}

\begin{proof}
Let $f \in L^{p,2}(\Cplus)$. Observe, for $\varphi \in C^{\infty}_c(\R)$, 
\begin{align*}
\left| \int_{-\infty}^\infty T_{\Cplus}f(x+iy)\,\varphi(x)\,dx \right| &\leq \sup_{K}|\varphi| \int_K |T_{\Cplus}f(x+iy)|\,dx,
\end{align*}
where $K  = \supp \varphi$. Since, for every $z \in \C$,  
\begin{align*}
T_{\Cplus}f(z)  &= -\frac{1}{\pi} \iint_{\Cplus} \frac{f(\zeta)}{\zeta - z}\, d\xi\,d\eta \\
&= -\frac{1}{\pi} \iint_{\C} \frac{g(\zeta)}{\zeta - z}\, d\xi\,d\eta \\
&= T_{\C}(g)(z),
\end{align*}
where 
\begin{numcases}{g(\zeta) :=}
    f(\zeta), & $\zeta \in \Cplus$ \nonumber\\
    0, & $\zeta \in\C\setminus\Cplus$, \nonumber
\end{numcases}
it follows that 
\[
\sup_{K}|\varphi| \int_K |T_{\Cplus}(f)(x+iy)|\,dx = \sup_{K}|\varphi| \int_K | T_{\C}(g)(x+iy)|\,dx
\]
Note $g \in L^{p,2}(\C)$ because
\[
\iint_D |g(z)|^p\,dx\,dy  = \iint_{D\cap \Cplus} |f(z)|^p\,dx\,dy   < \infty
\]
and 
\begin{align*}
\iint_D |(g)_2(z)|^p\,dx\,dy &= 
\iint_D \left| \frac{1}{|z|^2} \, g\left(\frac{1}{z}\right) \right|^p \,dx\,dy 
 = \iint_{D\cap \Cminus} \left| \frac{1}{|z|^2} \, f\left(\frac{1}{z}\right)\right|^p\,dx\,dy < \infty.
\end{align*}
By Theorem \ref{regofT}, $ T_{\C}(g) \in C^{0,\alpha}(\C)$. Since $K$ is compact and $ T_{\C}(g)$ is continuous, it follows that there exists $M$ such that $| T_{\C}(g)| \leq M$, for every $z \in K$. Hence, 
\begin{align*}
\sup_{K}|\varphi| \int_K | T_{\C}(g)(x+iy)|\,dx &\leq \sup_{K}|\varphi| M |K| < \infty.
\end{align*}
Since this holds for any $y > 0$, it follows that 
\[
\langle(T_{\Cplus}f)_b, \varphi \rangle := \lim_{y\searrow 0} \int_{-\infty}^\infty T_{\Cplus}f(x+iy)\,\varphi(x)\,dx
\]
exists, for every $\varphi \in C^\infty_c(\R)$.
\end{proof}

We will focus on a certain subset of $H(\Cplus)$ as the following theorem guarantees that they have boundary values in the sense of distributions.

\begin{theorem}[Theorem 3.1.11 \cite{hor}] \label{hormanderbv}
Let $I$ be an open interval on $\mathbb{R}$ and let 
\[
Z = \{z = x+ iy \in \mathbb{C} : x \in I, 0 < y < \gamma \}
\]
be a one sided complex neighborhood. If $f$ is holomorphic function in $Z$ such that for a non-negative integer $N$ 
\[
|f(x+iy)| \leq \frac{C}{y^N},\quad z \in Z,
\]
then $f$ has a boundary value in the sense of distributions $f_b$. 

\end{theorem}

\begin{deff}
Define the space $\mathcal{H}_{tg}$ as the functions in $H(\Cplus)$ with tempered growth at the boundary, i.e., $f$ in $H(\Cplus)$ such that there exists a positive integer $N$ and 
\[
|f(x+iy)| \leq \frac{C}{y^N},
\]
for all $z \in \Cplus$.
\end{deff}

\begin{prop}\label{ourversionofGHJH23point1}
Every $f \in \mathcal{H}_{tg}$ has a boundary value in the sense of distributions $f_b$ and is representable as 
\[
f(z) = \frac{1}{\pi} \langle f_b, P(x- \cdot, y) \rangle,
\]
where 
\[
P(x,y) = \frac{y}{x^2 + y^2}
\]
is the Poisson kernel for $\Cplus$.
\end{prop}

\begin{comm}
Note that Proposition \ref{ourversionofGHJH23point1} is inspired by Theorem 3.1 of \cite{GHJH2}. Also, it should be noted that $H^p(\Cplus) \subset \mathcal{H}_{tg}$, for every $0 < p \leq 1$. 
\end{comm}

\begin{proof}[Proof of Proposition \ref{ourversionofGHJH23point1}]

 Let $f \in \mathcal{H}_{tg}$, and let $\{y_k\}$ be a decreasing sequence of positive real numbers with the property that $y_k \to 0$ as $k \to \infty$. By Theorem \ref{hormanderbv}, $f$ has a boundary value in the sense of distributions $f_b$. Define $F_k(x,y) := f(x+i(y+y_k))$. By hypothesis, there exists a constant $C$ and a positive integer $N$ such that
 \[
 |F_k(x,y)| = |f(x_i(y+y_k))| \leq \frac{C}{|y+y_k|^N},
 \]
 so $F_k \in H^\infty(\Cplus)$, the space of bounded holomorphic functions on $\Cplus$. By Theorem 13.3 from \cite{Rep}, we have
\begin{align*}
F_k(x,y) &= \frac{1}{\pi} \int_{-\infty}^\infty F_k(t,0)\frac{y}{(x-t)^2 + y^2}\,dt \\
f(x+i(y +y_k)) &= \frac{1}{\pi} \langle F_k(\cdot, 0), P(x-\cdot,y)\rangle.
\end{align*}
Taking the limit as $k \to \infty$, we have
\begin{align*}
f(x+iy) &= \lim_{k\to\infty} \frac{1}{\pi} \langle F_k(\cdot, 0), P(x-\cdot, y) \rangle 
= \lim_{k\to\infty} \frac{1}{\pi} \int_{-\infty}^\infty F_k(t,0) P(x-t,y)\,dt\\
&=  \lim_{k\to\infty} \frac{1}{\pi} \int_{-\infty}^\infty f(t+iy_k) P(x-t,y)\,dt
= \frac{1}{\pi} \langle f_b, P(x-\cdot, y)\rangle.
\end{align*}
\end{proof}

\subsection{Schwarz Boundary Value Problem}

We prove the following theorem which extends the result of \cite{Gaertner} to a more general boundary condition.

\begin{theorem}\label{Schwarzupperhalf}
The Schwarz boundary value problem 
\begin{numcases}{}
\frac{\p w}{\p\z} = f \nonumber \\
\re{w_b} = \re{h_b} \nonumber\\
\im{w(i)} = c, \nonumber
\end{numcases}
for $f \in L^{p,2}(\Cplus)$, $p > 2$, $h \in  \mathcal{H}_{tg}$, and $c \in \mathbb{R}$, is solved by the function 
\begin{align*}
w(z) &= ic + \frac{1}{\pi}\langle h_b, P(x-\cdot, y)\rangle -  \frac{i}{\pi}\langle \im h_b, P(\cdot,1) \rangle\\
&\quad\quad -\frac{1}{\pi} \iint_{\Cplus} \left( f(\zeta)\left( \frac{1}{\zeta - z} - \frac{\zeta}{\zeta^2 + 1}\right) - \overline{f(\zeta)}\left( \frac{1}{\overline{\zeta}-z} - \frac{\overline{\zeta}}{\overline{\zeta}^2 + 1}\right)\right) \,d\xi\,d\eta\\
&=  ic + h(z) -  I + \mathcal{T}_{\Cplus}(f)(z),
\end{align*}
where $z = x+iy$, 
\[
P(x,y) = \frac{y}{x^2 + y^2}
\]
is the Poisson kernel for $\Cplus$, and 
\[
I = \frac{i}{\pi}\langle \im h_b, P(\cdot,1) \rangle.
\]
\end{theorem}

\begin{proof}
By Proposition \ref{ourversionofGHJH23point1}, if $h \in \mathcal{H}_{tg}$, then $h_b$ exists and 
\[
h(x+iy) = \frac{1}{\pi}\langle h_b, P(x-\cdot, y)\rangle.
\]
Since $ic, I,$ and $h$ are holomorphic and $\mathcal{T}_{\Cplus}(\cdot)$ is a right-inverse operator to $\frac{\p}{\p\z}$, it follows that $w$ solves $\frac{\p w}{\p\z} = f$. From \cite{Gaertner}, $\im\{\mathcal{T}_{\Cplus}(f)(i)\} = 0$. Hence, 
\begin{align*}
\im\{w(i)\} &= \im\left\{ic + \frac{1}{\pi} \langle h_b, P(\cdot, 1)\rangle  - I + \mathcal{T}_{\Cplus}(f)(i)\right\}  = c. 
\end{align*}
Since $\mathcal{T}_{\Cplus}(f)$ is purely imaginary on $\R$ and $ic$ and $I$ are purely imaginary everywhere, it follows that 
\[
\re\{ w_b \} = \re\{ ic + h_b - I + (\mathcal{T}_{\Cplus}(f))_b\} = \re\{h_b\}.
\]
\end{proof}

\begin{corr}\label{smallext}
The Schwarz boundary value problem 
\begin{numcases}{}
\frac{\p^n w}{\p\z^n} = f  \nonumber \\
\re\{w_b\} = \re\{(h_0)_b\} \nonumber \\
\re\{\frac{\p^k w}{\p\z^k}|_{\mathbb{R}}\}
= h_k, \quad 1\leq k \leq n-1 \nonumber\\
\im{\frac{\p^k w}{\p\z^k}(i)} = c_k,\quad 0 \leq k \leq n-1 \nonumber
\end{numcases}
for $f \in L^{p,2}(\Cplus)$, $p > 2$, $h_0 \in  \mathcal{H}_{tg}$, $t^k h_k \in L^p(\mathbb{R})\cap C(\mathbb{R})$, $1 \leq k \leq n-1$, and $\{c_k\}_{k=0}^{n-1} \subset \mathbb{R}$, is solved by the function 
\begin{align*}
w(z) &= ic + \frac{1}{\pi}\langle h_b, P(x-\cdot, y)\rangle -  \frac{i}{\pi}\langle \im h_b, P(\cdot,1) \rangle\\
&\quad +\sum_{k = 1}^{n-1}\frac{(-1)^k}{\pi i k!} \int_{-\infty}^\infty h_k(t)\left(\frac{1}{t-z} - \frac{t}{t^2+1}\right)(2t-z-\overline{z})^k\,dt \\
&\quad\quad + \frac{(-1)^n}{(n-1)!\pi}\iint_{\Cplus} \left( f(\zeta)\left( \frac{1}{\zeta - z} - \frac{\zeta}{\zeta^2 + 1}\right) \right.\\
&\quad\quad\quad - \left. \overline{f(\zeta)}\left( \frac{1}{\overline{\zeta}-z} - \frac{\overline{\zeta}}{\overline{\zeta}^2 + 1}\right)\right) (\zeta - z + \overline{\zeta - z})^{n-1}\,d\xi\,d\eta,
\end{align*}
\end{corr}

\begin{comm}
The above corollary extends Theorem 3.1 in \cite{bvpuhp} by generalizing the zero-order boundary condition. 
\end{comm}

We define a specific subset of $\mathcal{H}_{tg}$ and follow with an associated generalization of Corollary \ref{smallext}.

\begin{deff}
Denote by $\mathcal{H}_{tg,0}$ the subset of $f \in \mathcal{H}_{tg}$ with $\im\{f(i)\} = 0 $.
\end{deff}

\begin{theorem}\label{higherorderSchwarzupper}
The Schwarz boundary value problem 
\begin{numcases}{}
\frac{\p^n w}{\p\z^n} = f \nonumber\\
\re\left\{\left(\frac{\p^k w}{\p\z^k}\right)_b\right\}
= \re\{(h_k)_b\}, \quad 0\leq k \leq n-1\\
\im{w(i)} = c \nonumber\\
\im{\frac{\p^k w}{\p\z^k}(i)} = 0 ,\quad 1 \leq k \leq n-1, \nonumber
\end{numcases}
for $f \in L^{p,2}(\Cplus)$, $p > 2$, $h_0 \in \mathcal{H}_{tg}$, $h_k \in  \mathcal{H}_{tg, 0} \cap L^{p,2}(\Cplus)$, $1 \leq k \leq n-1$, and $c \in  \mathbb{R}$, is solved by the function 
\begin{align*}
w(z) &= ic + \frac{1}{\pi}\langle (h_0)_b, P(x-\cdot, y)\rangle -  \frac{i}{\pi}\langle \im (h_0)_b, P(\cdot,1) \rangle \\
&\quad +\sum_{k = 1}^{n-1} \left(\mathcal{T}_{\Cplus}^k \left(\frac{1}{\pi}\langle (h_k)_b, P(x-\cdot,y) \rangle \right)\right) + \mathcal{T}_{\Cplus}^n(f)(z).
\end{align*}
\end{theorem}

\begin{proof}
By direct computation and an appeal to Theorem \ref{higherupperth1}, $w$ solves the nonhomogeneous Cauchy-Riemann equation. By Theorem \ref{higherupperth2}, the boundary conditions are satisfied.
\end{proof} 

\begin{comm}
Note constants are not in the spaces $L^{p,2}(\Cplus)$, $p>2$. Under a direct iteration of Theorem \ref{Schwarzupperhalf}, $\mathcal{T}_{\Cplus}$ is applied to constants. The additional hypothesis on the higher-order boundary conditions in Theorem \ref{higherorderSchwarzupper} avoids the need for the operator $\mathcal{T}_{\Cplus}$ to be applied to constants. 
\end{comm}

\printbibliography
\end{document}